\documentclass[a4paper,10pt]{amsart}
\usepackage{amsmath,amssymb}
\usepackage{amscd}
\newtheorem{thm}{Theorem}[section]
\newtheorem{prop}[thm]{Proposition}
\newtheorem{lemma}[thm]{Lemma}
\newtheorem{cor}[thm]{Corollary}
\newtheorem{exam}[thm]{Example}
\theoremstyle{remark}
\newtheorem{remark}[thm]{Remark}
\newcommand{\id}{{\rm{id}}}
\newcommand{\Ad}{{\rm{Ad}}}

\newcommand{\BN}{\mathbf N}

\newcommand{\BC}{\mathbf C}

\newcommand{\BB}{\mathbf B}

\newcommand{\la}{\langle}
\newcommand{\ra}{\rangle}

\newcommand{\Rep}{{\rm{Rep}}}

\newcommand{\rCP}{{\rm{CP}}}

\newtheorem{Def}{Definition}[section]

\title{Strong Morita equivalences for completely positive linear maps and GNS-C*-correspondences}
\author{Kazunori Kodaka}
\address{Department of Mathematical Sciences, Faculty of Science, Ryukyu University, 
\endgraf
Nishihara-cho, Okinawa, 903-0213, Japan}
\address{\sl{E-mail address}: \rm{kodaka@math.u-ryukyu.ac.jp}}



\begin{document}
\begin{abstract}
We will consider the set of all completely positive linear maps from a unital $C^*$-algebra to the $C^*$-algebra
of all (bounded) adjointable right Hilbert $C^*$-module maps, which are automatically bounded, on a right
Hilbert $C^*$-module and we will introduce strong Morita equivalence for elements in this set. In this paper,
we will give the following result:
If two classes of two unital $C^*$-algebras are strongly Morita equivalent, respectively, 
then we can construct a bijective correspondence between two sets of all strong Morita equivalence classes of 
completely positive linear maps given as above. Furthermore, we will discuss the relation between strong Morita 
equivalence for completely positive linear maps and strong Morita equivalence for GNS-$C^*$-correspondences.
\end{abstract}

\maketitle

\section{Introduction}\label{sec:Intro} In the previous paper \cite {Kodaka:complete}, we introduce the notion of
strong Morita equivalence for comletely positive linear maps from a $C^*$-algebra to the $C^*$-algebra of
alll bounded linear operators on a Hilbert space using their minimal Stinespring representations and
\cite [Definition 2.1]{ER:multiplier} introduced by Echterhoff and Raeburn.
\par
In this paper, we consider completely positive linear maps from a unital $C^*$-algebra $C$ to the $C^*$-algebra
of all adjointable right Hilbert $B$-module maps, which are automatically bounded by Blackadar
\cite [Proposition 13.2.2]{Blackadar:K-Theory}, on a right Hilbert $B$-module $F$, where $B$ is a unital
$C^*$-algebra. We denote by $\rCP_B (C)$, the set of all such completely positive linear maps.
We introduce the notion of strong Morita equivalence for elements in $\rCP_B (C)$ using the KSGNS-representation
given in Lance \cite [Theorem 5.6]{Lance:toolkit}. In order to do this, we have to extend \cite [Definitoin 2.1]{ER:multiplier}
to non-degenerate representations of a unital $C^*$-algebra on a right Hilbert $C^*$-module. This
extension will be done in Section \ref{sec:Re}.
\par
In Section \ref{sec:Co}, we will give a result, which is similar to \cite [Corollary 6.6]{Kodaka:complete}
in the following way: Let $\psi$ be a completely positive linear map
from a uintal $C^*$-algebra $D$ to $\BB_B (F)$, where $B$ is a unital $C^*$-algebra and $\BB_B (F)$ is the
$C^*$-algebra of all adjointable rightable right Hilbert $B$-module maps on a right Hilbert $B$-module $F$.
Let $C$ be a unital $C^*$-algebra which is strongly Morita equivalent to $D$. Then in the same way as in
\cite [Section 6]{Kodaka:complete}, using $\psi$, we will construct a completely positive linear map $\phi$ from
$C$ to $\BB_B (F\otimes \BC^n )$, where $n$ is some positive integer. We will show that this construction induces
a bijective correspondence between $\rCP_B (D)/\! \sim$ and $\rCP_B (C)/\! \sim$, where $\rCP_B (D)/\! \sim$ and
$\rCP_B (C)/\! \sim$ denote the sets of all strong Morita equivalence classes in $\rCP_B (D)$ and $\rCP_B (C)$,
respectively. Furthermore, we assume that $A$ is strongly Morita equivalent to $B$. We will show that there exists
a bijective correspondence between $\rCP_B (C)/\! \sim$ and $\rCP_A (C)/\! \sim$. Thus, we obtain the following result:
If $C$ and $A$ are strongly Morita equivalent to $D$ and $B$, respectively, then there exists a bijective correspondence
between $\rCP_B (D)/\! \sim$ and $\rCP_A (C)/\!\sim$.
\par
In Section \ref{sec:GNS}, we will discuss the relation between strong Morita equivalence for completely positive
linear maps and strong Morita equivalence for GNS-$C^*$-correspondences induced by completely positive linear
maps. It is natural to do it since GNS-$C^*$-correspondences induced by completely positive linear maps are
constructed in the same way as in the KSGNS-construction.
\par
Let $A$ be a unital $C^*$-algebra and we denote by $\id_A$ and $1_A$,
the identity map on $A$ and the unit element in $A$, respectively. If no confusion arises, we simply denote them by
$\id$ and $1$, respectively.
\par
For each $n\in \BN$, let $M_n (A)$ be the $n\times n$-matrix algebra over $A$. We identify $M_n (A)$ with
$A\otimes M_n (\BC)$. Let $I_n$ be the unit element in $M_n (\BC)$. For each $a\in M_n (A)$, we denote by
$a_{ij}$ the $i\times j$-entry of $a$.
\par
Let $A$ and $B$ be unital $C^*$-algebras and $X$ an $A-B$-equivalence bimodule. We denote its left $A$-action
and right $B$-action on $X$ by $a\cdot x$ and $x\cdot b$ for any $a\in A$, $b\in B$, $x\in X$, respectively.
Let $\widetilde{X}$ be the dual $B-A$-equivalence bimodule of $X$ and let $\widetilde{x}$ denote the element
in $\widetilde{X}$ associated to the element $x\in X$.
\par
For right Hilbert $A$-modules $E$and $F$, let $\BB_A (E, F)$ be the space of all bounded adjointable right $A$-module
maps from $E$ to $F$ and if $E=F$, we denote $\BB_A (E,F)$ by $\BB_A (E)$. In this case, $\BB_A (E)$ is a unital
$C^*$-algebras.

\section{Representations of a unital $C^*$-algebra on right Hilbert modules and strong Morita equivalence}
\label{sec:Re} Let $C$ and $A$ be unital $C^*$-algebras.
\begin{Def}\label{def:Re1} We say that $(\pi_C , E)$ is a
\sl
representation
\rm
of $C$ on a right Hilbert $A$-module $E$
if $\pi_C$ is a homomorphism of $C$ to $\BB_A (E)$.
Furthermore, we that $(\pi_C , E)$ is
\sl
non-degenerate
\rm
if $\overline{\pi_C (C)E}=E$.
\end{Def}

Let $B$, $C$, $D$ be unital $C^*$-algebras. Let $(\pi_C , E)$ and $(\pi_D , F)$ be non-degenerate
representations of $C$ and $D$ on Hilbert $B$-modules $E$ and $F$, respectively.
\begin{Def}\label{def:Re2} We say that $(\pi_C , E)$ is
\sl
strongly Morita equivalent
\rm
to $(\pi_D , F)$ by if there exist a $C-D$-equivalence bimodule $Y$ and a linear map $\pi_Y$ from $Y$ to
$\BB_B(F, E)$ satisfying the following conditions:
\newline
(1) $\pi_Y (y)\pi_Y (z)^* =\pi_C ({}_C \la y , z \ra)$,
\newline
(2) $\pi_Y (y)^* \pi_Y (z)=\pi_D (\la y, z )_D )$,
\newline
(3) $\pi_Y (c\cdot y\cdot d)=\pi_C (c)\pi_Y (y)\pi_D (d)$
for any $y, z\in Y$, $c\in C$, $d\in D$.
\newline
We denote it by $(\pi_C , E)\sim (\pi_D , F)$.
\end{Def}
This definition is an extension of \cite [Definition2.1]{ER:multiplier} in the case that
$C^*$-algebras are unital.

\begin{prop}\label{prop:Re3} Strong Morita equivalence for non-degenerate representations of unital $C^*$-algebras
on right Hilbert $C^*$-modules is an equivalence relation.
\end{prop} 
\begin{proof} Let $(\pi_C , E)$ be a non-degenerate representation of a unital $C^*$-algebra $C$ on a right Hilbert
$B$-module $E$, where $B$ is a unital $C^*$-algebra. We regard $C$ as the trivial $C-C$-equivalence bimodule in
the natural way. We denote it by $Y_0$. Let $\pi_{Y_0}$ be the linear map from $Y_0$ to $\BB_B (E)$ defined by
$\pi_{Y_0}=\pi_C$. Then $\pi_{Y_0}$ satisfies Conditions (1)-(3) in Definition \ref{def:Re2}. Hence $(\pi_C , E)$ is
strongly Morita equivalent to itself.
\par
Let $(\pi_D ,F)$ be a non-degenerate representation of a unital $C^*$-algebra $D$ on a right Hilbert $B$-module
$F$. We suppose that $(\pi_C, E)$ is strongly Morita equivalent to $(\pi_D , F)$ with respect to
a $C-D$-equivalence bimodule $Y$ and a linear map $\pi_Y$ from $Y$ to $\BB_B (F, E)$. Let $\pi_{\widetilde{Y}}$ be
the linear map from $\widetilde{Y}$ to $BB_B (E, F)$ defined by
$$
\pi_{\widetilde{Y}}(\widetilde{y})=\pi_Y (y)^*
$$
for any $y\in Y$. Then by routine computations, we can see that $\pi_{\widetilde{Y}}$ satisfies Conditions (1)-(3) in
Definition \ref{def:Re2}. Thus we $(\pi_D , F)$ is strongly Morita equivalent to $(\pi_C , E)$.
\par
Let $(\pi_L , G)$ be a non-degenerate representation of a unital $C^*$-algebra $L$ on a right Hilbert $B$-module
$G$. We suppose that $(\pi_C ,E)$ is strongly Morita equivalent to $(\pi_D , F)$ with respect to a $C-D$-equivalence
bimodule $Y$ and a linear map $\pi_Y$ from $Y$ to $\BB_B (F, E)$ satisfying Conditions (1)-(3)
in Definition \ref{def:Re2}. Also, we suppose that $(\pi_D , F)$ is strongly Morita equivalent to $(\pi_L , G)$
with respect to a $D-L$-equivalence bimodule $W$ and a linear map $\pi_W$ from $W$ to $\BB_B(G, F)$
satisfying Conditions (1)-(3) in Definition \ref{def:Re2}. We show that $(\pi_C , E)$ is strongly Morita equivalent to
$(\pi_L , G)$. Clearly $C$ is strongly Morita equivalent to $L$ with respect to the $C-L$-equivalence bimodule
$Y\otimes_D W$. Let $\pi_{Y\otimes_D W}$ be the linear map from $Y\otimes_D W$ to $\BB_B (G, E)$ defined by
$$
\pi_{Y\otimes_D W}(y\otimes w)=\pi_Y (y)\pi_W (w)
$$
for any $y\in Y$, $w\in W$. Then by routine computations, we can see that $\pi_{Y\otimes_D W}$ satisfies Conditions
(1)-(3) in Definition \ref{def:Re2}. Therefore, we obtain the conclusion.
\end{proof}

Let $C$, $A$ be unital $C^*$-algebras. For $i=1, 2$, let $(\pi_i , E_i )$ be a non-degenerate representation of $C$
on a right Hilbert $A$-module $E_i$.

\begin{lemma}\label{lem:Re4} With the above notation, if $(\pi_1 , E_1 )$ and $(\pi_2 , E_2 )$ are unitarily equivalent,
they are strongly Morita equivalent.
\end{lemma}
\begin{proof} Since $(\pi_1 , E_1 )$ and $(\pi_2 , E_2 )$ are unitarily equivalent, there is an isometry
$u\in \BB_B (E_1 , E_2 )$ such that $\pi_2 =\Ad(u)\circ\pi_1$. Let $Y_0 (=C)$ be the trivial $C-C$-equivalence
bimodule. Let $\pi_{Y_0}$ be the linear map from $Y_0$ to $\BB_B (E_2 , E_1 )$ defined by
$$
\pi_{Y_0}(y)=\pi_1 (y)u^*
$$
Then $\pi_{Y_0}$ satisfies Conditions (1)-(3) in Definition \ref{def:Re2}. Thus, $(\pi_1 , E_1 )$ and $(\pi_2 , E_2 )$
are strongly Morita equivalent.
\end{proof}
Let $C, D$ and $B$ be unital $C^*$-algebras and we suppose that $C$ and $D$ are strongly Morita equivalent with
respect to a $C-D$-equivalence bimodule $Y$. Let $(\pi_D , F)$ be a non-degenerate representation of $D$ on a
right Hilbert $B$-module $F$. Modifying \cite [Section 2]{ER:multiplier}, we construct a non-degenerate
representation $(\pi_C , E)$ of $C$ on a right Hilbert $B$-module $E$, which is strongly Morita equivalent to
$(\pi_D , F)$. We regard $F$ as a Hilbert $D-B$-bimodule using the representation $(\pi_D , F)$.
Let $E$ be a right Hilbert $B$-module by setting
$$
E=Y\otimes_D F .
$$
Let $\pi_C$ be the map from $C$ to $\BB_B (E)$ defined by
$$
\pi_C (c)(y\otimes f)=(c\cdot y )\otimes f
$$
for any $c\in C$, $y\in Y$, $f\in F$. Then by \cite [Proposition 1.7(i)]{BMS:quasi} and
easy computations, $(\pi_C , E)$ is a non-degenerate representation of $C$. We will
show that $(\pi_C , E)$ is strongly Morita equivalent to $(\pi_D , F)$.

\begin{lemma}\label{lem:Re5} With the above notation, $(\pi_C , E)$ is strongly Morita equivalent to $(\pi_D , F)$.
\end{lemma}
\begin{proof} It suffices to show that there is a linear map $\pi_Y$ from $Y$ to $\BB_B (F, E)$
satisfying Conditions (1)-(3) in Definition \ref{def:Re2}. Let $\pi_Y$ be the linear map from $Y$
to the space of all bounded right $B$-module maps from $F$ to $E$ defined by
$$
\pi_Y (y)f=y\otimes f
$$
for any $y\in Y$, $f\in F$. We show that $\pi_Y (y)\in\BB_B (F, E)$ for any $y\in Y$ and that
$\pi_Y$ satisfies Conditions (1)-(3) in Definition \ref{def:Re2}.
For any $y, z \in Y$, $f, g \in F$.
\begin{align*}
\la z\otimes f\, , \, \pi_Y (y)g \ra_B & =\la z\otimes f\, ,\, y\otimes g \ra_B =\la f\, , \, \pi_D (\la z, y \ra_D )g \ra_B \\
& =\la \pi_D (\la y , z \ra_D )f\, , \, g \ra_B .
\end{align*}
Thus $\pi(y)\in \BB_B (F,E)$ for any $y\in Y$ and 
$\pi(y)^* (z\otimes f)=\pi_D (\la y, z \ra_D )f$ for any $y, z\in Y$, $f\in F$. Then for any $y, y_1 , z\in Y$, $f\in F$,
$$
\pi_Y (y)\pi_Y (y_1 )^* (z\otimes f)=\pi_Y (y)\pi_D (\la y_1 \, , \, z \ra_D )f
=y\otimes \pi_D (\la y_1 \, , \, z \ra_D )f .
$$
On the other hand,
$$
\pi_C ({}_C \la y, y_1 \ra )(z\otimes f) =({}_C \la y\, , \, y_1 \ra \cdot z )\otimes f =y\cdot \la y_1 \, , \, z \ra_D \otimes f
=y\otimes \pi_D (\la y_1 \, , \, z \ra_D )f .
$$
Hence $\pi_Y (y)\pi_Y (y_1 )^* =\pi_C ({}_C \la y\, , \, y_1 \ra )$. Also, for any $y, y_1 \in Y$, $f\in F$,
$$
\pi_Y (y)^* \pi_Y (y_1 )f =\pi_Y (y)^* (y_1 \otimes f )=\pi_D (\la y\, , \, y_1 \ra_D )f .
$$
Hence $\pi_Y (y)^* \pi_Y (y_1 )=\pi_D (\la y\, , \, y_1 \ra_D )$. Furthermore, for any $c\in C$, $d\in D$, $f\in F$, $y\in Y$,
$$
\pi_Y (c\cdot y\cdot d)f=c\cdot y\cdot d\otimes f=\pi_C (c)(y\otimes\pi_D (d)f)
=\pi_C (c)\pi_Y (y)\pi_D (d)f .
$$
Thus $\pi_Y (c\cdot y\cdot d)=\pi_C (c)\pi_Y (y)\pi_D (d)$. Therefore, we obtain the conclusion.
\end{proof}

We call the above $(\pi_C , E)$ a non-degenerate representation
\sl
induced by
\rm
$(\pi_D , F)$ and $Y$.
\par
Let $B, C$ and $A$ be unital $C^*$-algebras. Let $\Rep_B (C)$ be the set of all non-degenerate representations
of $C$ on right Hilbert $B$-modules. Let $\Rep_B (C)/\!\sim$ be the set of
all strong Morita equivalence classes of elements in $\Rep_B (C)$. Similarly, we define $\Rep_A (C)$ and
$\Rep_A (C)/\!\sim$, respectively. For any $(\pi, E)\in\Rep_A (C)$, $[\pi, E]$ denotes the strong Morita
equivalence class of $(\pi, E)$. We construct a bijective correspondence $\Phi$ between $\Rep_B (C)/\!\sim$ and
$\Rep_A (C)/\!\sim$.
\par
Let $(\pi, F)\in\Rep_B (C)$ and let $Z$ be a $B-A$-equivalence bimodule.
We consider the right Hilbert $A$-module $F\otimes_B Z$. Let $\pi^Z$ be the map from $C$ to
$\BB_A (F\otimes_B Z)$ defined by
$$
\pi^Z (c)(f\otimes z)=\pi(c)f\otimes z
$$
for any $f\in F$, $z\in Z$, $c\in C$. Then by easy computations, $(\pi^Z , F\otimes_B Z)$ is a non-degenerate
representation of $C$ on $F\otimes_B Z$. Thus $(\pi^Z , F\otimes_B Z)\in\Rep_A (C)$. Let $\Phi$ be the map
from $\Rep_B (C)/\!\sim$ to $\Rep_A (C)/\!\sim$ defined by
$$
\Phi([\pi, F])=[\pi^Z , F\otimes_B Z]
$$
for any $[\pi, F]\in\Rep_B (C)/\!\sim$.

\begin{lemma}\label{lem:Re6} With the above notation, $\Phi$ is well-defined and a bijective correspondence between
$\Rep_B (C)/\!\sim$ and $\Rep_A (C)/\!\sim$.
\end{lemma}
\begin{proof} First, we show that $\Phi$ is well-defined. Let $(\pi, F), (\pi_1  ,  F_1 )$ be elements in $\Rep_B (C)$
with $[\pi, F]=[\pi_1 , F_1 ]$ in $\Rep_B (C)/\!\sim$. Then there are a $C-C$-equivalence bimodule $Y$ and a linear map
$\pi_Y $ from $Y$ to $\BB_B (F, F)$ satisfying Conditions (1)-(3) in Definition \ref{def:Re2}. Let $(\pi^Z , F\otimes_B Z)$
and $(\pi_1^Z , F_1 \otimes_B Z)$ be non-degenerate representations of $C$ induced by $(\pi, F)$, $Y$ and
$(\pi_1 , F_1 )$, $Y$ respectively. Let $\pi_Y^Z$ be the linear map from $Y$ to
$\BB_A (F_1 \otimes_B Z\, , \, F\otimes_B Z )$ defined by
$$
\pi_Y^Z (f\otimes z)=\pi_Y (f)\otimes z =(\pi_Y \otimes\id_Z )(f\otimes z)
$$
for any $f\in F$, $z\in Z$. By easy computations, $\pi_Y^Z$ satisfies Conditions (1)-(3) in Definition \ref{def:Re2}.
Thus, $\Phi$ is well-defined. Next we show that $\Phi$ is a bijective correspondence between $\Rep_B (C)/\!\sim$ and
$\Rep_A (C)/\!\sim$. Let $\Psi$ be the map from $\Rep_A (C)/\!\sim$ to $\Rep_B (C)/\!\sim$ defined by
$$
\Psi([\pi, E])=[\pi^{\widetilde{Z}}, E\otimes_A \widetilde{Z}]
$$
for any $(\pi, E)\in\Rep_A (C)$. Then for any $(\pi, E)\in\Rep_A (C)$,
$$
(\Phi\circ\Psi)([\pi, E])=\Phi([\pi^{\widetilde{Z}}\, , \, E\otimes_A \widetilde{Z}])
=[(\pi^{\widetilde{Z}})^Z \, , \, Z\otimes_A \widetilde{Z}\otimes_B Z] .
$$
We show that $[(\pi^{\widetilde{Z}})^Z \, , \, E\otimes_A \widetilde{Z}\otimes_B Z]=[\pi, E]$ in $\Rep_A(C)/\!\sim$.
Let $Y_0 (=C)$ be the trivial $C-C$-equivalence bimodule. Let $\pi_{Y_0}$ be the linear map from $Y_0$ to
the space of all bounded right $A$-module maps from $E\otimes_A \widetilde{Z}\otimes_A Z$ to $E$
defined by
$$
\pi_{Y_0}(y)(e\otimes\widetilde{z}\otimes z_1 )=\pi(y)(e\cdot \la z, z_1 \ra_A )
$$
for any $e\in E$, $z, z_1 \in Z$. First, we claim that $\pi_{Y_0}(y)$ is adjointable for any $y\in Y_0$.
Indeed, for any $y\in Y_0$, $e, e_1 \in E$, $z, z_1 \in Z$,
\begin{align*}
\la e\, , \, \pi_{Y_0}(y)(e_1 \otimes\widetilde{z}\otimes z_1 ) \ra_A & =\la e\, , \, \pi(y)e_1 \cdot \la z, z_1 \ra_A \ra_A 
= \la \pi(y)^* e\, , \, e_1 \cdot \la z, z_1 \ra_A \ra_A \\
& =\la \pi(y)^* e \, , \, e_1 \ra_A \la z , z_1 \ra_A .
\end{align*}
Since $Z$ is full with the right $A$-valued inner product, there is a finite subset $\{u_i \}_{i=1}^n$ of $Z$ such
that $\sum_{i=1}^n \la u_i , u_i \ra_A =1$ by Kajiwara and Watatani \cite [Corollary 1.19]{KW1:bimodule}.
Then
\begin{align*}
\la \sum_{i=1}^n \pi(y)^* e\otimes\widetilde{u_i}\otimes u_i \, , \, e_1 \otimes z\otimes z_1 \ra_A
& =\sum_{i=1}^n \la \widetilde{u_i}\otimes u_i \, , \, \la \pi(y)^* e , e_1 \ra_A \cdot \widetilde{z}\otimes z \ra_A \\
& =\sum_{i=1}^n \la \widetilde{u_i }\otimes u_i \, , \, [z\cdot\la e_1 , \pi(y)^* e \ra_A ]^{\widetilde{}}\otimes z_1 \ra_A \\
& =\sum_{i=1}^n \la u_i \, , \, \la \widetilde{u_i}\, , \,
 [z\cdot \la e_1 , \pi(y)^* e \ra_A ]^{\widetilde{}}\ra_B \cdot z_1 \ra_A \\
& =\sum_{i=1}^n \la u_i \, , \, {}_B \la u_i\, , \,
 [z\cdot \la e_1 , \pi(y)^* e \ra_A ]\ra \cdot z_1 \ra_A \\
 & =\sum_{i=1}^n \la {}_B \la [z\cdot \la e_1 \, , \, \pi(y)^* e \ra_A ] \, , \, u_i \ra \cdot u_i \, , \, z_1 \ra_A \\
 & =\la \pi(y)^* e \, , \, e_1 \ra_A \la z , z_1 \ra_A .
\end{align*}
Thus $\pi_{Y_0}(y)$ is adjointable and $\pi_{Y_0}(y)^* =\sum_{i=1}^n \pi(y)^* e\otimes\widetilde{u_i}\otimes u_i $ for any
$y\in Y_0$. Hence $\pi_{Y_0}(y)\in \BB_A (E\otimes_A \widetilde{Z}\otimes_B Z , E)$ for
any $y\in Y_0$ and $\pi_{Y_0}$ is a linear map from
$Y_0$ to $\BB_A (E\otimes_A \widetilde{Z}\otimes_B Z)$. Next, we show that $\pi_{Y_0}$ satisfies Conditions (1)-(3) in
Definition \ref{def:Re2}. For any $y, y_1 \in Y_0$, $e\in E$,
\begin{align*}
\pi_{Y_0}(y)\pi_{Y_0}(y_1 )^* e & =\sum_{i=1}^n \pi_{Y_0}(y)\pi(y_1 )^* \otimes\widetilde{u_i}\otimes u_i 
=\sum_{i=1}^n \pi(y)\pi(y_1 )^* e\cdot \la u_i , u_i \ra_A \\
& =\pi({}_C \la y , y_1 \ra)e .
\end{align*}
Thus $\pi_{Y_0}(y)\pi_{Y_0}(y_1 )^* =\pi({}_C \la y , y_1 \ra)$. Also, for any $y, y_1 \in Y_0$, $e\in E$, $z, z_1 \in Z$,
\begin{align*}
\pi_{Y_0}(y)^* \pi_{Y_0}(y_1 )(e\otimes\widetilde{z}\otimes z_1 ) & =\pi_{Y_0}(y)^* \pi(y_1 )e\cdot \la z, z_1 \ra_A \\
& =\sum_{i=1}^n \pi(y)\pi(y_1 )^*e\cdot \la z, z_1 \ra_A\otimes \widetilde{u_i}\otimes u_i \\
& =\sum_{i=1}^n \pi(y)^* \pi(y_1 )e\otimes [u_i \cdot \la z_1 , z \ra_A ]^{\widetilde{}}\otimes u_i \\
& =\sum_{i=1}^n \pi(y^* y_1 )e\otimes[{}_B \la u_i \, , \, z_1 \ra\cdot z]^{\widetilde{}}\otimes u_i \\
& =\sum_{i=1}^n \pi(y^* y_1 )e\otimes \widetilde{z}\otimes{}_B \la z_1 , u_i \ra \cdot u_i \\
& =\pi(y^* y_1 )e\otimes\widetilde{z}\otimes z_1 \\
& =(\pi^{\widetilde{Z}})^Z (\la y, y_1 \ra_C ) .
\end{align*}
Furthermore, for any $y\in Y_0$, $c, d\in C$, $e\in E$, $z, z_1 \in Z$,
$$
\pi_{Y_0}(c\cdot y\cdot d)(e\otimes\widetilde{z}\otimes z_1 )=\pi(cyd)e\cdot \la z, z_1 \ra_A .
$$
On the other hand,
\begin{align*}
\pi(c)\pi_{Y_0}(y)(\pi^{\widetilde{Z}})^Z (d)(e\otimes\widetilde{z}\otimes z_1 ) & =
\pi(c)\pi_{Y_0}(y)(\pi(d)e\otimes\widetilde{z}\otimes z_1 ) \\
& =\pi(c)(\pi(yd)e\cdot \la z, z_1 \ra_A ) \\
& =\pi(cyd)e\cdot \la z, z_1 \ra_A
\end{align*}
since $\pi(c)$ is a right Hilbert $A$-module map. It follows that
$[(\pi^{\widetilde{Z}})^Z \, , \, E\otimes_A \widetilde{Z}\otimes_B Z]=[\pi, E]$
in $\Rep_A (C)/\!\sim$. Thus $\Phi\circ\Psi=\id$ on $\Rep_A (C)/\!\sim$.
Similarly, $\Psi\circ\Phi=\id$ on $\Rep_B (C)/\!\sim$. Therefore, 
$\Phi$ is a bijective correspondence between $\Rep_B (C)/\!\sim$ and $\Rep_A (C)/\!\sim$.
\end{proof}

\section{A correspondence of strong Morita equivalence classes of completely positive linear maps on
unital $C^*$-algebras}\label{sec:Co} Following Lance \cite [Chapter 5]{Lance:toolkit}, we give
the following theorem(KSGNS construction):

\begin{thm}\label{thm:Co1}$($\cite [Theorem 5.6]{Lance:toolkit}$)$ Let $D$ nad $B$ be unital $C^*$-algebras and let
$F$ be a right Hilbert $B$-module. Let $\psi$ be a completely positive linear map from $D$ to
$\BB_B (F)$ .
\newline
$(1)$ There exista right Hilbert $B$-module $F_{\psi}$, a representation $\pi_{\psi}$ of $D$ on $F_{\psi}$ and
an elemet $V_{\psi}$ in $\BB_B (F, F_{\psi})$ such that
$$
\psi(d)=V_{\psi}^* \pi_{\psi}(d)V_{\psi}
$$
for any $d\in D$ and such that $\overline{\pi_{\psi}(D)V_{\psi}F}=F_{\psi}$.
\newline
$(2)$ Let $G$ be a right Hilbert $B$-module, $\pi$ a representation of $D$ on $G$ and $W\in\BB_B (F, G)$.
We suppose that $\psi(d)=W^* \pi(d)W$ for any $d\in D$ and that $\overline{\pi(D)WF}=G$.
Then there is an isomety $U$ from $F_{\psi}$ onto $G$ such that
$$
\pi(d)=U\pi_{\psi}(d)U^*
$$
for any $d\in D$.
\end{thm}
\begin{proof}
This is immediate by \cite[Theorem 5.6]{Lance:toolkit}.
\end{proof}

\begin{remark}\label{remark:Co2} The representation $(\pi_{\psi}, F_{\psi})$ is non-degenerate since
$$
F_{\psi}=\overline{\pi_{\psi}(D)V_{\psi}F}\subset \overline{\pi_{\psi}(D)F_{\psi}}\subset F_{\psi} .
$$
Also, the representation $(\pi, G)$ is non-degenerate.
\end{remark}

We call the above $(\pi_{\psi}, V_{\psi}, F_{\psi})$ a
\sl
minimal
\rm
KSGNS-representation for $\psi$.
\par
Let $C, D$ and $B$ be unital $C^*$-algebras and let $F$ be a right Hilbert $B$-module.
Let $\phi$ and $\psi$ be completely positive linear maps from $C$ and $D$ to $\BB_B(F)$, respectively.

\begin{Def}\label{def:Co3} We say that $\phi$ is
\sl
strongly Morita equivalent
\rm
to $\psi$ if a minimal KSGNS-representation for $\phi$ is strongly Morita equivalent to that for $\psi$.
We denote it by $\phi\sim\psi$.
\end{Def}

\begin{prop}\label{prop:Co4} Strong Morita equivalence for completely positive linear maps from
unital $C^*$-algebras to the $C^*$-algebras of all adjointable right Hilbert modules maps on right
Hilbert modules is an equivalence relation.
\end{prop}
\begin{proof} This is immediate by Proposition \ref{prop:Re3}, Lemma \ref{lem:Re4} and
Theorem \ref{thm:Co1}.
\end{proof}

Let $C$, $D$ and $B$ be unital $C^*$-algebras. We suppose that $C$ and $D$ are strongly Morita equivalent
with respect to a $C-D$-equivalence bimodule $Y$. Let $F$ be a right Hilbert $B$-module and let $\psi$ be a
completely positive linear map from $D$ to $\BB_B (F)$. Let $(\pi_C , E)$ be the non-degenerate representation of
$C$ on $E$ induced by $(\pi_{\psi}, V_{\psi}, F_{\psi})$ and $Y$. Let $\{u_i \}_{i=1}^n$ be a finite subset of $Y$.
We construct a linear map $\phi$ from $C$ to $\BB_B (F)\otimes M_n (\BC)$ in the same way as in
\cite [Section 6]{Kodaka:complete}, that is, let $\phi$ be the linear map from $C$ to $\BB_B (F)\otimes M_n (\BC)$
defined by
$$
\left[ \phi(c)_{i j} \right]_{i, j=1}^n =\left[\psi (\la u_i \, , \, c\cdot u_j \ra_D )\right]_{i, j=1}^n
$$
for any $c\in C$. In the same way as in the proof of \cite [Lemma 6.1]{Kodaka:complete}, we can see that
$\phi$ is a completely positive linear map from $C$ to $\BB_B (F)\otimes M_n (\BC)$. Let $F\otimes\BC^n$ be the
algebraic tensor product $F$ and $\BC^n$ and we regard $F\otimes \BC^n$ as a right  Hilbert $B$-module
in the natural way. We identify $\BB_B (F)\otimes M_n (\BC )$ with $\BB_B (F\otimes\BC^n )$. Also, we identify
$F\otimes\BC^n$ with $\oplus_1^n F$ as right Hilbert $B$-modules. Let $(\pi_{\phi}, Y_{\phi}, E_{\phi})$ be a
minimal KSGNS-representation for $\phi$. We show that $(\pi_C , E)$ is unitarily equivalent to
$(\pi_{\phi}, V_{\phi}, E_{\phi})$. 
\par
Let $C\odot(F\otimes\BC^n )$ be the algebraic tensor product of $C$ and $F\otimes \BC^n$. We define
a map $U$ from $C\odot(F\otimes\BC^n )$ to $E$ by setting
$$
U(c\otimes f\otimes\lambda)=\sum_{i=1}^n \lambda_i (c\cdot u_i )\otimes 1_D \otimes f
$$
for any $c\in C$, $f\in F$, $\lambda=\left[ \begin{array}{c}
\lambda_1 \\
\vdots \\
\lambda_n \end{array}\right] \in\BC^n$ and extending linearly. Then in the same way as in the proof of
\cite [Lemma 6.2]{Kodaka:complete}, we can see that $U$ is an isometry from $C\odot (F\otimes \BC^n )$
to $E$. Hence we can extend $U$ to an isometry from $E_{\phi}$ to $E$. We denote it the same symbol
$U$. From now on, we assume that $\{u_i \}_{i=1}^n$ is a left $C$-basis in $Y$. Such a finite subset
$\{u_i \}_{i=1}^n$ of $Y$ exists by Kajiwara and Watatani \cite [Corollary 1.19]{KW1:bimodule}.
Then we can see that $U$ is surjective in the same way as in the proof of \cite [Lemma 6.3]{Kodaka:complete}.
Thus, $U$ is an isometry from $E_{\phi}$ onto $E$. Furthermore, in the same way as in
\cite [Lemma 6.4]{Kodaka:complete},
we can see that $\pi_{\phi}(c)=U^* \pi_C (c)U$ for any $c\in C$. Therefore, we obtain the following
theorem:

\begin{thm}\label{thm:Co5} Let $C$, $D$ and $B$ be unital $C^*$-algebras. We suppose that $C$ and $D$
are strongly Morita equivalent with respect to a $C-D$-equivalence bimodule $Y$. Let $\{u_i \}_{i=1}^n$
be a left $C$-basis in $Y$. Let $\psi$ be a completely positive linear map from $D$ to $\BB_B (F)$, where
$F$ is a right Hilbert $B$-module. Let $\phi$ be a mp from $C$ to $\BB_B( F\otimes\BC^n )$ defined by
$$
\left[\phi(c)_{i j} \right]_{i, j=1}^n =\left[ \psi(\la u_i , c\cdot u_j \ra_D )\right]_{i, j=1}^n
$$
for any $c\in C$. Then $\phi$ is a completely positive linear map from $C$ to $\BB_B (F\otimes\BC^n )$,
which is strongly Morita equivalent to $\psi$.
\end{thm}

For any $\psi\in\rCP_B (D)$, we denote by $\mathcal{F}_0 (\psi)$ the element in $\rCP_B (C)$ defined in the above
and $\mathcal{F}_0$ denotes the map
$$
\rCP_B (D)\to \rCP_B (C): \psi\mapsto\mathcal{F}_0(\psi) .
$$
By the above theorem, we can obtain the map
$$
\rCP_B (D)/\!\sim \to \rCP_B (C)/\!\sim \, : [\psi]\mapsto[\mathcal{F}_0(\psi)] ,
$$
where $[\psi]$ and $[\mathcal{F}_0 (\psi)]$ are the strong Morita equivalence classes of $\psi$ and
$\mathcal{F}_0 (\psi)$, respectively. We denote by the symbol $\mathcal{F}$ the map from
$\rCP_B (D)/\!\sim$ to $\rCP_B (C)/\!\sim$. In the same as above, we can also obtain the map from
$\rCP_B (C)/\!\sim$ to $\rCP_B (D)/\!\sim$,  which is the inverse map of $\mathcal{F}$. 
Hence $\mathcal{F}$ is a bijective correspondence between $\rCP_B (D)/\!\sim$ and $\rCP_B (C)/\!\sim$.
Therefore, we obtain the following:

\begin{cor}\label{cor:Co6} Let $C, D, B$ be
unital $C^*$-algebras. Then $\mathcal{F}$ is a bijective correspondence between $\rCP_B (D)/\!\sim$
and $\rCP_B (C)/\!\sim$.
\end{cor}

Let $C, A$ and $B$ be unital $C^*$-algebras. We suppose that $A$ and $B$ are strongly Morita
equivalent with respect to a $B-A$-equivalence bimodule $Z$. Let $\phi$ be a completely positive
liear map from $C$ to $\BB_B (F)$, where $F$ is a right Hilbert $B$-module. Let $\phi^Z$ be the linear map
from $C$ to $\BB_A (F\otimes_B Z)$ defined by
$$
\phi^Z (c)(f\otimes z)=\phi(c)(f)\otimes z =(\phi\otimes\id_Z )(c)(f\otimes z)
$$
for any $c\in C$, $f\in F$, $z\in Z$. Also, let $(\pi_{\phi}, V_{\phi}, F_{\phi})$ be a minimal KSGNC-representation
for $\phi$. Let $(\pi_{\phi}^Z ,F_{\phi}\otimes_B Z)$ be the non-degenerate representation of $C$
on the right Hilbert $A$-module $F_{\phi}\otimes_B Z$ defined in Section \ref{sec:Re}.

\begin{lemma}\label{lem:Co7} With the above notation, for any $c\in C$,
$$
\phi^Z (c)=(V_{\phi}^* \otimes\id_Z )\pi_{\phi}^Z (c)(V_{\phi}\otimes\id_Z )
$$
and
$$
\overline{\pi_{\phi}^Z (C)(V_{\phi}\otimes\id_Z )(F\otimes_B Z)}=F_{\phi}\otimes_B Z .
$$
\end{lemma}
\begin{proof}
For any $c\in C$, $f\in F$, $z\in Z$,
\begin{align*}
(V_{\phi}^* \otimes\id_Z )\pi_{\phi}^Z (c)(V_{\phi}\otimes\id_Z )(f\otimes z) &=(V_{\phi}^* \otimes\id_Z )\pi_{\phi}^Z (c)(V_{\phi}f\otimes z) \\
& =(V_{\phi}^* \otimes\id_Z )(\pi_{\phi}(c)V_{\phi}F \otimes z) \\
& =V_{\phi}^* \pi_{\phi}(c)V_{\phi}F \otimes z \\
& =\phi(c)(f)\otimes z \\
& =\phi^Z (c)(f\otimes z) .
\end{align*}
Hence
$$
\phi^Z (c)=(V_{\phi}^* \otimes\id_Z )\pi_{\phi}^Z (c)(V_{\phi}\otimes\id_Z )
$$
for any $c\in C$ and clearly $V_{\phi}\otimes\id\in \BB_A (F\otimes_B Z \, , \, F_{\phi}\otimes_B Z )$.
Thus $\phi^Z$ is completely positive by \cite [Proposition 5.5]{Lance:toolkit}.
Also,
$$
\overline{\pi_{\phi}^Z (C)(V_{\phi}\otimes\id_Z )(F\otimes_B Z)}=
\overline{\pi_{\phi}(C)V_{\phi}F\otimes_B Z}=\overline{F_{\phi}\otimes_B Z}
$$
since $\overline{\pi_{\phi}(C)V_{\phi}F}=F_{\phi}$.
\end{proof}

By Lemma \ref{lem:Co7}, $(\pi_{\phi}^Z , V_{\phi}\otimes\id_Z, F_{\phi}\otimes_B Z )$ is a minimal 
KSGNS-representation for $\phi^Z$.
\par
We denote by $\widetilde{\Phi}$, the map $[\phi]\in\rCP_B (C)/\!\sim\,\mapsto [\phi^Z ]\in\rCP_A (C)/\!\sim$,
where $[\phi]$ and $[\phi^Z ]$ are the strong Morita equivalence classes of $\phi\in\rCP_B (C)$ and
$\phi^Z \in \rCP_A (C)$, respectively.

\begin{lemma}\label{lem:Co8} With the above notation, $\widetilde{\Phi}$ is well-defined and $\widetilde{\Phi}$ is a
bijective correspondence between $\rCP_B (C)/\!\sim$ and $\rCP_A (C)/\!\sim$.
\end{lemma}
\begin{proof}
This is immediate by Lemmas \ref{lem:Re6} and \ref{lem:Co7}.
\end{proof}

We give a similar result to \cite [Corollary 6.6]{Kodaka:complete}.

\begin{thm}\label{thm:Co9} Let $C$, $A$ and $D$, $B$ be unital $C^*$-algebras.
We suppose that $C$ and $A$ are strongly Morita equivalent to $D$ and $B$, respectively.
Then there exists a bijective correspondence between $\rCP_B (D)/\!\sim$ and $\rCP_A (C)/\!\sim$.
\end{thm}
\begin{proof} By Corollary \ref{cor:Co6}, there is a bijective correspondence between $\rCP_B (D)/\!\sim$ and
$\rCP_B (C)/\!\sim$. Also, by Lemma \ref{lem:Co8}, there is a bijective correspondence
between $\rCP_B (C)/\!\sim$ and $\rCP_A (C)/\!\sim$. Thus, we obtain a bijective correspondence between
$\rCP_B (D)/\!\sim$ and $\rCP_A (C)/\!\sim$.
\end{proof}

\section{GNS-$C^*$-correspondences and strong Morita equivalence}\label{sec:GNS}
Let $C, D$ and $A, B$ be unital $C^*$-algebras. Let $\phi$ be a completely positive linear map
from $C$ to $A$. In the same way as the KSGNS-construction, we can construct a $C-A$-correspondence
$E_{\phi}$. Following Marrero and Muhly \cite {MM:CP-algebra},
we call $E_{\phi}$ a
\sl
GNS-$C^*$-correspondence
\rm
induced by $\phi$. Also, let $\psi$ be a completely positive linear map from $D$ to $B$ and
let $E_{\psi}$ be the GNS-$C^*$-correspondence induced by $\psi$, which is a $D-B$-correspodence.
Following Muhly and Solel \cite {MS:tensor}, we give the following definition:

\begin{Def}\label{def:GNS0} We say that $E_{\phi}$ is
\sl
strongly Morita equivalent to $E_{\psi}$
\rm
if there exist a $C-D$-equivalence bimodule $Y$ and an $A-B$-equivalence bimodule $X$
such that $E_{\phi}\otimes_A X$ is isomorphic to $Y\otimes_D E_{\psi}$ as $C-B$-correspondences,
that is, there is a linear map $\Phi$ from $E_{\phi}\otimes_A X$ onto $Y\otimes_D E_{\psi}$
satisfying the following:
\newline
(1) $\Phi(c\cdot x\cdot b)=c\cdot \Phi(x)\cdot b$,
\newline
(2) $\la \Phi(x), \Phi(y) \ra_B =\la x, y \ra_B$,
\newline
for any $x, y\in E_{\phi}\otimes_A X$, $c\in C$, $b\in B$.
\end{Def}

Let $\phi$ be a completely positive
linear map from $C$ to $A$ and $E_{\phi}$ the $C-A$-correspondence induced by $\phi$.
We regard $A$ as the trivial right Hilbert $A$-module in the natural way and $\phi$ can be regarded
as a completely positive linear map from 
$C$ to $\BB_A (A)$. Also $E_{\phi}$ can be regarded as a right Hilbert $A$-module in a minimal
KSGNS-representation for $\phi$.
\par
Let $\psi$ be a completely positive linear map from $D$ to $B$ and $E_{\psi}$ the $D-B$-correspondence
induced by $\psi$. We also regard $E_{\psi}$ as a right Hilbert $B$-module in a minimal KSGNS-representation for
$\psi$.
\par
We suppose that there is an $A-B$-equivalence bimodule $X$. Then $A\cong \BB_B (X)$ as $C^*$-algebras by
the isomorphism $T$ of $A$ onto $\BB_B (X)$ defined by
$$
T_a (x)=a\cdot x
$$
for any $a\in A$, $x\in X$. Let $\phi_T =T\circ \phi$, a completely positive linear map from $C$ to $\BB_B (X)$.
First, we show that
if $\phi_T $ and $\psi$ are strongly Morita equivalent as completely positive linear maps,
then $E_{\phi}$ and $E_{\psi}$ are strongly Morita equivalent as GNS-$C^*$-correspondences.
Since $X$ is an $A-B$-equivalence bimodule and $B$ is unital, there is a finite subset
$\{v_j \}_{j=1}^m$ of $X$ such that
$$
\sum_{j=1}^m \la v_j , v_j \ra_B =1_B
$$
by the proof of \cite [Corollary 1.19]{KW1:bimodule}. Let $(\pi_{\phi_T}, V_{\phi_T}, E_{\phi_T})$ be
a minimal KSGNS-representation for $\phi_T$.

\begin{lemma}\label{lem:GNS1} With the above notation, let $\Psi$
be a linear map from $E_{\phi}\otimes_A X$ to $E_{\phi_T}$ defined by
$$
\Psi((c\otimes a)\otimes x)=c\otimes(a\cdot x)
$$
for any $a\in A$, $c\in C$, $x\in X$. Then $\Psi$ is a $C-B$-correspondence isomorphism
of $E_{\phi}\otimes_A X$ onto $E_{\phi_T}$. Furthermore, $\Psi^{-1}$ is given by
$$
\Psi^{-1}(c\otimes x)=\sum_{j=1}^m c\otimes{}_A \la x, v_j \ra\otimes v_j
$$
for any $c\in C$, $x\in X$, where $E_{\phi_T}$ is regarded as a $C-B$-correspondence.
\end{lemma}
\begin{proof} Let $c, c_1 \in C$, $a, a_1\in A$, $x, x_1 \in X$. Then
\begin{align*}
\la \Psi((c\otimes a)\otimes x \, , \, (c_1 \otimes a_1 )\otimes x_1 \ra_B & =
\la c\otimes a\cdot x \, , \, c_1 \otimes a_1 \cdot x_1 \ra_B \\
& =\la a\cdot x \, ,\, \phi_T (c^* c_1 )(a_1 \cdot x) \ra_B \\
& =\la x \, , \, a^* \phi(c^* c_1 )a_1 \cdot x \ra_B .
\end{align*}
On the other hand,
\begin{align*}
\la (c\otimes a )\otimes x \, , \, (c_1 \otimes a_1 )\otimes x_1 \ra_B & =
\la x \, , \la c\otimes a\, , \, c_1 \otimes a_1 \ra_A \cdot x_1 \ra_B \\
& =\la x\, , \, \la a \, , \, \phi(c^* c_1 )a_1 \ra_A \cdot x_1 \ra_B \\
& =\la x\, , \, a^* \phi(c^* c_1 )a_1 \cdot x_1 \ra_B .
\end{align*}
Hence $\Psi$ preserves the right $B$-valued inner products. Also, let $c, c_1 \in C$, $a\in A$, $x\in X$, $b\in B$.
Then
\begin{align*}
\Psi(c_1 \cdot (c\otimes a)\otimes x\cdot b) & =\Psi((c_1 c\otimes a)\otimes (x\cdot b)) \\
& =c_1 c\otimes (a\cdot x\cdot b) \\
& =c_1 \cdot (c\otimes a\cdot x)\cdot b \\
& =c_1 \cdot \Psi((c\otimes a )\otimes x)\cdot b .
\end{align*}
Thus, $\Psi$ is a $C-B$-bimodule map. Furthermore, by easy computations,
$\Psi\circ\Psi^{-1}=\id_{E_{\phi_T}}$ and $\Psi^{-1}\circ\Psi=\id_{E_{\phi}\otimes_A X}$.
Therefore, we obtain the conclusion.
\end{proof}

Let  $(\pi_{\psi}, V_{\psi}, E_{\psi})$ be a minimal
KSGNS-representations for $\psi$. We also regard $E_{\psi}$ as a $D-B$-correspondence for $\psi$.
We suppose that $\phi_T$ and $\psi$ are
strongly Morita equivalent as completely positive linear maps with respect to a $C-D$-equivalence
bimodule $Y$ and a linear map $\pi_Y$ from $Y$ to $\BB_B (E_{\psi}, E_{\phi_T})$ satisfying
Conditions (1)-(3) in Definition \ref{def:Re2}. 

\begin{lemma}\label{lem:GNS2} With the above notation, let $\Phi$ be the linear map from $Y\otimes_D E_{\psi}$
to $E_{\phi_T}$ defined by
$$
\Phi(y\otimes(d\otimes b))=\pi_Y (y)(d\otimes b)
$$
for any $y\in Y$, $d\in D$, $b\in B$. Then $\Phi$ is a $C-B$-correspondence isomorphism of $Y\otimes_D E_{\psi}$
onto $E_{\phi_T}$.
\end{lemma}
\begin{proof}
Let $y, y_1 \in Y$, $d, d_1 \in D$, $b, b_1 \in B$. Then
\begin{align*}
\la \Phi(y\otimes(d\otimes b)) \, , \, \Phi(y_1 \otimes(d_1 \otimes b_1 ))\ra_B & =
\la \pi_Y (y)(d\otimes b) \, , \, \pi_Y (y_1 )(d_1 \otimes b_1 ) \ra_B \\
& =\la [\pi_Y (y_1 )^* \pi_Y (y)](d\otimes b ) \, , \, d_1 \otimes b_1 \ra_B \\
& =\la \pi_{\psi}(\la y_1 \, , \, y \ra_D )(d\otimes b) \, , \, d_1 \otimes b_1 \ra_B \\
& =\la \la y_1 \, , \, y \ra_D d \otimes b \, , \, d_1 \otimes b_1 \ra_B .
\end{align*}
On the other hand,
\begin{align*}
\la y\otimes (d\otimes b) \, , \, y_1 \otimes (d_1 \otimes b_1 ) \ra_B 
& =\la d\otimes  b\, , \, \la y, y_1 \ra_D d_1 \otimes b_1 \ra_B \\
& =\la \la y_1 \, , \, y \ra_D d \otimes b \, , \, d_1 \otimes b_1 \ra_B .
\end{align*}
Hence $\Phi$ preserves the right $B$-valued inner products.
Let $c\in C$, $y\in Y$, $d\in D$, $b, b_1 \in B$. Then
\begin{align*}
\Phi(c\cdot y\otimes (d\otimes b)\cdot b_1 ) & =
\Phi((c\otimes y)\otimes (d\otimes bb_1 )) \\
& =\pi_Y (c\cdot y)(d\otimes bb_1 ) \\
& =(\pi_{\phi}(c)\pi_Y (y))(d\otimes b)\cdot b_1 \\
& =c\cdot \pi_Y (y)(d\otimes b)\cdot b_1 \\
& =c\cdot \Phi(y\otimes(d\otimes b))\cdot b_1 .
\end{align*}
Hence $\Phi$ is a $C-B$-bimodule map from $Y\otimes_D E_{\psi}$ to $E_{\phi_T}$.
Furthermore, since $Y$ is a
$C-D$-equivalence bimodule and $C$ is unital, there is a finite subset $\{u_i \}_{i=1}^n$ of
$Y$ such that
$$
\sum_{i=1}^n {}_C \la u_i \, , \, u_i \ra =1_C .
$$
Let $c\in C$, $x\in X$. Then $\sum_{i=1}^n u_i \otimes\pi_Y (u_i )^* (c\otimes x)\in Y\otimes_D E_{\psi}$ and
\begin{align*}
\Phi(\sum_{i=1}^n u_i \otimes\pi_Y (u_i )^* (c\otimes x))
& =\sum_{i=1}^n \pi_Y (u_i )(\pi_Y (u_i )^* (c\otimes x)) \\
& =\sum_{i=1}^n (\pi_Y (u_i )\pi_Y (u_i )^* )(c\otimes x) \\
& =\sum_{i=1}^n \pi_{\phi}({}_C \la u_i \, ,\, u_i \ra)(c\otimes x) \\
& =c\otimes x
\end{align*}
since $\pi_Y (u_i )\pi_Y (u_i )^* =\pi_{\phi} ({}_C \la u_i \, , \, u_i \ra)$ for $i=1,2,\dots, n$. 
Thus
$\Phi$ is surjective. Therefore, $\Phi$ is a $C-B$-correspondence isomorphism of
$Y\otimes_D E_{\psi}$ onto $E_{\phi_T}$. 
\end{proof}
\par
Next, we suppose that $E_{\phi}$ and $E_{\psi}$ are strongly Morita equivalent as
GNS-$C^*$-correspondences. Then there are
a $C-D$-equivalence bimodule $Y$, $A-B$-equivalence bimodule $X$
and a $C-B$-correspondence isomorphism $\Phi_1$ of $Y\otimes_D E_{\psi}$
onto $E_{\phi}\otimes_A X$. Let $\Psi$ be the $C-B$-correspondence isomorphism
of $E_{\phi}\otimes_A X$ onto $E_{\phi_T}$ defined in Lemma \ref{lem:GNS1}.
Let $\Phi=\Psi\circ\Phi_1$, a $C-B$-correspondence isomorphism of $Y\otimes_D E_{\psi}$
onto $E_{\phi_T}$.

\begin{lemma}\label{lem:GNS3} With the above notation, let $\pi_Y$ be the
space of all right $B$-module maps from $E_{\psi}$ to $E_{\phi}$
defined by
$$
\pi_Y (y)(d\otimes b)=\Phi(y\otimes(d\otimes b))
$$
for any $y\in Y$, $d\in D$, $b\in B$. Then $\pi_Y$ is a linear map from $Y$ to $\BB_B (E_{\psi}, E_{\phi_T})$
satisfying Conditions $(1)-(3)$ in Definition $\ref{def:Re2}$.
\end{lemma}
\begin{proof} First, we show that $\pi_Y (y)\in \BB_B (E_{\psi}, E_{\phi_T})$ for any $y\in Y$.
It suffices to show that $\pi_Y (y)$ is adjointable for any $y\in Y$. For any $y, y_1 \in Y$, $d, d_1 \in D$, $b, b_1 \in B$,
\begin{align*}
\la \pi_Y (y)(d\otimes b) \, , \, \pi_Y (y_1 )(d_1 \otimes b_1 ) \ra_B & =
\la \Phi(y\otimes (d\otimes b)) \, , \, \Phi(y_1 \otimes (d_1 \otimes b_1 )) \ra_B \\
& =\la y\otimes (d\otimes b) \, , \, y_1 \otimes (d_1 \otimes b_1 ) \ra_B \\
& =\la d\otimes b \, , \, \la y , y_1 \ra_D d_1 \otimes b_1 \ra_B .
\end{align*}
On the other hand,
\begin{align*}
\la \pi_{\psi}(\la y_1 \, , \, y \ra_D )(d\otimes b) \, , \, d_1 \otimes b_1 \ra_B & =
\la \la y_1 , y \ra_D d \otimes b \, , \, d_1 \otimes b_1 \ra_B \\
& =\la d\otimes b \, , \, \la y , y_1 \ra_D d_1 \otimes b_1 \ra_B .
\end{align*}
Since the linear span of the set
$$
\{\Phi(y\otimes (d\otimes b)) \, | \, y\in Y, \quad d\in D, \quad b\in B \}
$$
is dense in $E_{\phi_T}$, $\pi_Y (y_1 )$ is adjointable for any $y\in Y$ and
$$
\pi_Y (y_1 )^* (\Phi(y\otimes (d\otimes b)))=\pi_{\psi}(\la y_1 ,  y \ra_D )(d\otimes b) .
$$
And,
$$
\pi(y_1 )^* \pi(y)=\pi_{\psi}(\la y_1 , y \ra_D ) .
$$
Furthermore, for any $y, y_1 , z\in Y$, $d\in D$, $b\in B$,
$$
\pi_Y (z)\pi_Y (y_1 )^* \Phi(y\otimes (d\otimes b))=\pi_Y (z)(\la y_1 , y \ra_D d\otimes b)
=\Phi(z\otimes \la y_1 , y \ra_D d\otimes b) .
$$
On the other hand,
\begin{align*}
\pi_{\phi_T}({}_C \la z, y_1 \ra )\Phi(y\otimes (d\otimes b)) & =
{}_C \la z, y_1 \ra \cdot \Phi(y\otimes(d\otimes b)) \\
& =\Phi(({}_C \la z, y_1 \ra \cdot y)\otimes(d\otimes b)) \\
& =\Phi(z\cdot \la y_1 , y \ra_D \otimes (d\otimes b )) \\
& =\Phi(z\otimes (\la y_1 , y \ra_D d\otimes b )) .
\end{align*}
Hence $\pi_Y (z)\pi_Y (y_1 )^* =\pi_{\phi_T}({}_C \la z, y_1 \ra)$ for any $z, y_1 \in Y$.
Finally, for any $c\in C$, $y\in Y$, $b\in B$, $d, d_1 \in D$,
\begin{align*}
\pi_{\phi_T}(c)\pi_Y (y)\pi_{\psi}(d_1 )(d\otimes b) & =\pi_{\phi_T}(c)\pi_Y (y)(d_1 d\otimes b) \\
& =\pi_{\phi_T}(c)\Phi(y\otimes (d_1 d \otimes b)) \\
& =c\cdot \Phi(y\otimes (d_1 d \otimes b) )\\
& =\Phi((c\cdot y\cdot d_1 )\otimes (d\otimes b) )\\
& =\pi_Y (c\cdot y\cdot d_1 )(d\otimes b) .
\end{align*}
Thus $\pi_{\phi_T}(c)\pi_Y (y)\pi_{\psi}(d_1 )=\pi_Y (c\cdot y\cdot d_1 )$ for any $c\in C$, $y\in Y$, $d_1 \in D$.
Therefore, we obtain the conclusion.
\end{proof}

Combining Lemmas \ref{lem:GNS1}, \ref{lem:GNS2} and \ref{lem:GNS3},
we obtain the following theorem:

\begin{thm}\label{thm:GNS4} Let $C, D$ and $A, B$ be unital $C^*$-algebras. Let $\phi$ and
$\psi$ be completely positive linear maps from $C$ and $D$ to $A$ and $B$, respectively and $E_{\phi}$ and
$E_{\psi}$ the $C-A$-correspondence and the $D-B$-correspondence induced by $\phi$ and $\psi$,
respectively. Then the following hold:
\newline
$(1)$ If there is an $A-B$-equivalence bimodule $X$ and $\phi_T$ and $\psi$ are strongly Morita equivalent
as completely positive linear maps, then $E_{\phi}$ and $E_{\psi}$ are strongly Morita equivalent
as GNS-$C^*$-correspondences, where $\phi_T$ is a completely positive linear map
from $C$ to $\BB_B (X)$ induced by $\phi$ and the isomorphism $T$ of $A$ onto $\BB_B (X)$ defined by
$T_a (x)=a\cdot x$ for any $a\in A$, $x\in X$.
\newline
$(2)$ If $E_{\phi}$ and $E_{\psi}$ are strongly Morita equivalent
as GNS-$C^*$-correspondences, then there is an $A-B$-equivalence bimodule $X$ and $\phi_T$ and
$\psi$ are strongly Morita equivalent
as completely positive linear maps,
where $\phi_T$ is a completely positive linear map
from $C$ to $\BB_B (X)$ induced by $\phi$ and the isomorphism $T$ of $A$ onto $\BB_B (X)$ defined by
$T_a (x)=a\cdot x$ for any $a\in A$, $x\in X$.
\end{thm}
\begin{proof} We can show (1) by Lemmas \ref{lem:GNS1} and \ref{lem:GNS2}. We can also show
(2) by Lemma \ref{lem:GNS3}.
\end{proof}

Finally, we give an example on Theorem \ref{thm:GNS4}.

\begin{exam}\label{exam:GNS5} Let $C, D, B$ be unial $C^*$-algebras and $\psi$ a completely
positive linear map from $D$ to $B$. We regard $\psi$ as a completely positive linear map
from $D$ to $\BB_B (B)$. We assume that $C$ and $D$ are strongly Morita equivalent.
Then we can construct $\phi_T$, a completely positive linear map from $C$ to $\BB_B (B\otimes\BC^n )$
by Theorem $\ref{thm:Co5}$, where $n$ is some positive integer. Let $T$ be the isomorphism of
$B\otimes M_n (\BC)$ onto $\BB_B (B\otimes\BC^n )$ defined by
$$
T_b (x)=b\cdot x
$$
for any $b\in B\otimes M_n (\BC )$, $x\in B\otimes\BC^n$, where we regard $B\otimes\BC^n $
as a $B\otimes M_n (\BC)-B$-equivalence bimodule in the natural way. Let $\phi=T^{-1}\circ\phi_T$.
Then by Theorems $\ref{thm:Co5}$ and $\ref{thm:GNS4}$, $E_{\phi}$ and $E_{\psi}$ are strongly Morita
equivalent as GNS-$C^*$-correspondences, where $E_{\phi}$ and $E_{\psi}$ are
a $C-B\otimes M_n (\BC)$-correspondence and a $D-B$-correspondence, respectively.
\end{exam}

\end{document}